\documentclass[12pt]{article}%
\usepackage[intlimits]{amsmath}
\usepackage{amssymb}
\usepackage[T1]{fontenc}
\usepackage[sc]{mathpazo}
\usepackage{color}
\usepackage[colorlinks]{hyperref}
\usepackage{amsfonts}
\usepackage{graphicx}%
\definecolor {refcol}{RGB}{40,0,255}
\hypersetup{colorlinks=true,allcolors=refcol}
\makeatletter
\newfont{\footsc}{cmcsc10 at 8truept}
\newfont{\footbf}{cmbx10 at 8truept}
\newfont{\footrm}{cmr10 at 10truept}
\newtheorem{theorem}{Theorem}

\newtheorem{conjecture}[theorem]{Conjecture}

\newtheorem{problem}[theorem]{Problem}
\newtheorem{proposition}[theorem]{Proposition}
\newtheorem{question}[theorem]{Question}

\newenvironment{proof}[1][Proof]{\noindent{\textbf {#1}  }}  {\hfill$\Box$\bigskip}

\begin{document}

\title{\textbf{Remarks} \textbf{on the energy of regular graphs}}
\author{V. Nikiforov\thanks{Department of Mathematical Sciences, University of
Memphis, Memphis TN 38152, USA}}
\date{}
\maketitle

\begin{abstract}
The energy of a graph is the sum of the absolute values of the eigenvalues of
its adjacency matrix. This note is about the energy of regular graphs.  It is
shown that graphs that are close to regular can be made regular with a
negligible change of the energy. Also a $k$-regular graph can be extended to a
$k$-regular graph of a slightly larger order with almost the same energy. As
an application, it is shown that for every sufficiently large $n,$ there
exists a regular graph $G$ of order $n$ whose energy $\left\Vert G\right\Vert
_{\ast}$ satisfies
\[
\left\Vert G\right\Vert _{\ast}>\frac{1}{2}n^{3/2}-n^{13/10}.
\]
Several infinite families of graphs with maximal or submaximal energy are
given, and the energy of almost all regular graphs is
determined.\textit{\medskip}

\textbf{Keywords: }\textit{graph energy; regular graph; random regular graph;
degree deviation; maximal energy graphs.\medskip}


\end{abstract}

\section{Introduction}

In \cite{Gut78} Gutman introduced the energy of a graph as the sum of the
absolute values of the eigenvalues of its adjacency matrix. Since the energy
of a graph $G$ is the trace norm of its adjacency matrix, we write $\left\Vert
G\right\Vert _{\ast}$ for the energy of $G.$

In this note we discuss the energy of regular graphs, an area that has been
studied (see, e.g., \cite{GFPR07} and \cite{GLS12}), but which is still vastly unexplored.

In the groundbreaking paper \cite{KoMo01}, Koolen and Moulton showed that
\begin{equation}
\left\Vert G\right\Vert _{\ast}\leq k+\sqrt{k\left(  n-k\right)  \left(
n-1\right)  } \label{KM}%
\end{equation}
for every $k$-regular graph of order $n.$ Equality holds in (\ref{KM}) if and
only if $G=K_{n},$ or $G=\left(  n/2\right)  K_{2},$ or $G$ is a strongly
regular graph with parameters $\left(  n,k,a,a\right)  ,$ i.e., a $k$-regular
design graph (see, e.g.,\ \cite{BJK99}, p. 144). Since bound (\ref{KM}) is
precise for an amazing variety of graphs, in \cite{Bal04} Balakrishnan
proposed to study how tight this bound is in general, and asked the following
relevant question, which is quoted here verbatim:

\begin{question}
[Balakrishnan \cite{Bal04}]\label{bq}For any two positive integers $n$ and
$k$, $n-1>k\geq2$ and $\varepsilon>0$, does there exist a $k$-regular graph
$G$ with $\left\Vert G\right\Vert _{\ast}/B_{2}>1-\varepsilon,$ where
$B_{2}=k+\sqrt{k\left(  n-1\right)  \left(  n-k\right)  }.$
\end{question}

Unfortunately, despite its sound idea, Question \ref{bq} is incoherent in the
above form, because the quantifier of $\varepsilon$ is unclear, and if $n$ and
$k$ do not depend on $\varepsilon,$ the answer is almost always negative.
These weaknesses have been exploited in the literature to trivialize
Balakrishnan's question, e.g., in \cite{LLS10}. Moreover, important recent
results of van Dam, Haemers, and Koolen \cite{DHK14} imply that the original
question of Balakrishnan's cannot be preserved without excising many
combinations of $n$ and $k.$ To clarify this point, we state an essential
corollary of the paper \cite{DHK14}, which, however, has eluded its authors:

\begin{theorem}
\label{th1}Let $k\geq2$ and $n\geq k^{2}-k+1.$ If $G$ is a $k$-regular graph
of order $n,$ then
\begin{equation}
\left\Vert G\right\Vert _{\ast}\leq\left(  \sqrt{k-1}+\frac{1}{k+\sqrt{k-1}%
}\right)  n. \label{DHK}%
\end{equation}
Equality holds in (\ref{DHK}) if and only if $G$ is a disjoint union of
incidence graphs of projective planes of order $k-1$ or $k=2$ and $G$ is
disjoint union of triangles and hexagons.
\end{theorem}

Note that for $n>k^{2}-k+1$ the right side of (\ref{KM}) is greater than the
right side of (\ref{DHK}), so Theorem \ref{th1} is a clear improvement over
(\ref{KM}). Therefore, if we want to investigate $k$-regular graphs for which
(\ref{KM}) is almost tight, we must suppose that $k\geq\sqrt{n}$.

With this hindsight, it seems natural to split Question \ref{bq} into two
conjectures: First, study Question \ref{bq} for $k$-regular graphs whenever
$k$ grows not too fast with $n$, say, $k=o\left(  n\right)  .$ For this case
we venture the following conjecture:

\begin{conjecture}
\label{con0}For every $\varepsilon>0,$ there exist $\delta>0$ and
$k_{0}\left(  \varepsilon\right)  $ such that if $n\delta>k>k_{0}\left(
\varepsilon\right)  $ and $kn$ is even, there exists a $k$-regular graph $G$
of order $n$ with
\[
\left\Vert G\right\Vert _{\ast}\geq\left(  1-\varepsilon\right)  \sqrt{k}n.
\]

\end{conjecture}

Second, study Question \ref{bq} for dense regular graphs, for which the
following conjecture might hold:

\begin{conjecture}
\label{mcon}For every $\varepsilon>0$ and $n$ sufficiently large, if
$n>k>\sqrt{n}$ and $kn$ is even, there exists a $k$-regular graph $G$ of order
$n$ with
\begin{equation}
\left\Vert G\right\Vert _{\ast}\geq\left(  1-\varepsilon\right)
\sqrt{k\left(  n-k\right)  n}. \label{mineq}%
\end{equation}

\end{conjecture}

As it turns out that the important factor in Conjecture \ref{mcon} is not the
degree of $G$, but its density. Thus, it turns out that Conjecture \ref{mcon}
follows from the following simpler one:

\begin{conjecture}
\label{mcon1}Let $0<c\leq1/2.$ For every $\varepsilon>0,$ if $n$ is
sufficiently large. there exists a graph $G$ of order $n$ with at most
$cn^{2}/2$ edges such that
\begin{equation}
\left\Vert G\right\Vert _{\ast}\geq\left(  1-\varepsilon\right)
\sqrt{c\left(  1-c\right)  }n^{3/2}. \label{mineq1}%
\end{equation}

\end{conjecture}

First, note that in Conjecture \ref{mcon1} the requirement $0<c\leq1/2$ is not
restrictive and is just a convenience, for the energy of a graph is roughly
the same as the energy of its complement, as shown in Section \ref{EC} below.

However, the crucial implication here is that if a graph $G$ of order $n$ has
at most $cn^{2}/2$ edges and satisfies (\ref{mineq1}), then it must be close
to regular, and such graphs can be made regular with a negligible loss of
energy, easily concealed by the $\left(  1-\varepsilon\right)  $ coefficient
in (\ref{mineq}).

As another illustration of these ideas, we shall improve a result on graphs of
maximal energy: Recall that in \cite{Nik07j}, a question of Koolen and Moulton
\cite{KoMo01} was answered by showing that if $n$ is sufficiently large, there
exists a graph $G$ of order $n$ with
\[
\left\Vert G\right\Vert _{\ast}>\frac{1}{2}n^{3/2}-n^{11/10}.
\]

In Section, \ref{DD} we show that $G$ can be chosen regular at a negligible
loss of energy:

\begin{theorem}
\label{th2}If $n$ is sufficiently large, there exists a regular graph $G$ of
order $n$ with%
\[
\left\Vert G\right\Vert _{\ast}>\frac{1}{2}n^{3/2}-n^{13/10}.
\]

\end{theorem}

The remaining part of this note covers the following topics: In Section
\ref{DD}, we establish relations between the energy and the degree deviation
of a graph, from which we deduce Theorem \ref{th2}. Section \ref{EC} is
dedicated to inequalities between the energy of a graph and the energy of its
complement. In Section \ref{ESR}, we use some design graphs to show the
tightness of bound (\ref{KM}) for infinite classes of graphs. Finally, in
Section \ref{ERR}, we determine the energy of almost all regular graphs, which
so far seems to have gone unnoticed.

\section{\label{DD}The energy of graphs and degree deviation}

The principal goal of this section is to show that graphs that are close to
regular can be made regular with just a minor change of the energy. Another
goal is to show that a $k$-regular graph can be extended to a $k$-regular
graph of a slightly larger order with almost the same energy. These two
results can be used in various graph constructions related to the energy of
regular graphs.

Let $G$ be a graph of order $n$ and size $m.$ In \cite{Nik06}, the author
suggested to use the function%
\[
s\left(  G\right)  =\sum_{u\in V\left(  G\right)  }\left\vert d\left(
u\right)  -\frac{2m}{n}\right\vert
\]
as a measure of irregularity of $G.$ Here $d\left(  u\right)  $ stands for the
degree of the vertex $u.$ Clearly, $G$ is regular if and only if $s\left(
G\right)  =0,$ so we say that $G$ is close to regular if $s\left(  G\right)
=o\left(  n^{2}\right)  .$

The following theorem was proved in \cite{Nik06}:\medskip\ 

\textbf{Theorem R }\emph{For every graph }$G$\emph{ of order }$n$\emph{ and
size }$m,$\emph{ there exists a graph }$R$\emph{ of order }$n$\emph{ and size
}$m$\emph{ such that }$\Delta\left(  R\right)  \leq\delta\left(  R\right)
+1$\emph{ and }$R$\emph{ differs from }$G$\emph{ in at most }$s\left(
G\right)  $\emph{ edges. In particular, if }$2m/n$\emph{ is an integer then
}$R$\emph{ is }$\left(  2m/n\right)  $\emph{-regular.\medskip}

Applications of Theorem R often use the following simple bound:

\begin{proposition}
\label{D}Let $H$ and $G$ be two graphs on the same vertex set. If $H$ differs
from $G$ in at most $m$ edges, then
\[
\left\vert \left\Vert H\right\Vert _{\ast}-\left\Vert G\right\Vert _{\ast
}\right\vert \leq\sqrt{2mn}.
\]

\end{proposition}

\begin{proof}
Indeed, write $A\left(  G\right)  $ and $A\left(  H\right)  $ for the
adjacency matrices of $G$ and $H.$ Clearly the matrix $A\left(  G\right)
-A\left(  H\right)  $ has at most $2m$ entries of modulus one. Hence,
\[
\left\Vert A\left(  G\right)  -A\left(  H\right)  \right\Vert _{\ast}\leq
\sqrt{n\sum\lambda_{i}^{2}\left(  A\left(  G\right)  -A\left(  H\right)
\right)  }=\sqrt{2mn},
\]
as claimed.
\end{proof}

Theorem R, Proposition \ref{D}, and the triangle inequality for the trace norm
help to construct regular graphs from irregular ones keeping control on the
energy change:

\begin{proposition}
\label{pro4}Let $n>k\geq1$ and $nk$ be even. If $G$ is a graph of order $n$
with $kn/2$ edges, then there exists a $k$-regular graph $R$ of order $n$ such
that
\[
\left\vert \left\Vert R\right\Vert _{\ast}-\left\Vert G\right\Vert _{\ast
}\right\vert \leq\sqrt{2s\left(  G\right)  n}%
\]

\end{proposition}

Here we outline a concrete application of this proposition, which is one of
the main results of this section:

\begin{theorem}
\label{th3}Let $n>t>k\geq2$ and suppose that $nk$ is even. If $H$ is a
$k$-regular graph of order $t,$ there exists a $k$-regular graph $G$ of order
$n$ such that
\[
\left\vert \left\Vert H\right\Vert _{\ast}-\left\Vert G\right\Vert _{\ast
}\right\vert <3\sqrt{\left(  n-t\right)  kn}%
\]

\end{theorem}

\begin{proof}
Write $H_{0}$ for the graph of order $n$ obtained by adding $n-t$ isolated
vertices to $H$ and note that $\left\Vert H_{0}\right\Vert _{\ast}=\left\Vert
H\right\Vert _{\ast}.$ Let $G_{0}$ be the graph of order $n$ obtained by
joining every new vertex of $H_{0\text{ }}$to $\left\lfloor k/2\right\rfloor $
or $\left\lceil k/2\right\rceil $ of the vertices of $H$ so that $G_{0}$ has
$nk/2$ edges. It is not hard to see that
\begin{equation}
s\left(  G_{0}\right)  =\left(  n-t\right)  k. \label{in3}%
\end{equation}
Since $G_{0}$ has $\left(  n-t\right)  k/2$ edges in addition to those of $H,$
Proposition \ref{D} implies that
\begin{equation}
\left\vert \left\Vert G_{0}\right\Vert _{\ast}-\left\Vert H_{0}\right\Vert
\right\vert \leq\sqrt{\left(  n-t\right)  kn}. \label{in4}%
\end{equation}
On the other hand, Theorem R implies that there exists a $k$-regular graph $G$
of order $n$ such that $G$ differs from $G_{0}$ in at most $s\left(
G_{0}\right)  $ edges. Hence, Proposition \ref{pro4}, together with
(\ref{in3}), implies that
\[
\left\vert \left\Vert G\right\Vert _{\ast}-\left\Vert G_{0}\right\Vert _{\ast
}\right\vert \leq\sqrt{2s\left(  G_{0}\right)  n}\leq\sqrt{2\left(
n-t\right)  kn}.
\]
Therefore, in view of (\ref{in4}), we find that
\begin{align*}
\left\vert \left\Vert G\right\Vert _{\ast}-\left\Vert H\right\Vert _{\ast
}\right\vert  &  \leq\left\vert \left\Vert G_{0}\right\Vert _{\ast}-\left\Vert
H_{0}\right\Vert \right\vert +\left\vert \left\Vert G\right\Vert _{\ast
}-\left\Vert G_{0}\right\Vert _{\ast}\right\vert \\
&  \leq\sqrt{2\left(  n-t\right)  kn}+\sqrt{\left(  n-t\right)  kn}%
<3\sqrt{\left(  n-t\right)  kn},
\end{align*}
as claimed.
\end{proof}

Armed with Theorem \ref{th3}, we shall encounter no difficulty in proving
Theorem \ref{th2}:

\begin{proof}
[\textbf{Proof of Theorem \ref{th2}}]Recall that for $n$ sufficiently large,
there exists a prime $p$ such that $p\equiv1\operatorname{mod}$ $4$ and $p\leq
n+n^{11/20+\varepsilon}$ (see, e.g., \cite{BHP97}, Theorem 3). Suppose that
$n$ is large enough and fix some prime $p\leq n+n^{3/5}/8$ such that
$p\equiv1$ $\operatorname{mod}$ $4.$ Write $G_{p}$ for the Paley graph of
order $p.$ Recall that $V\left(  G_{p}\right)  :=\left\{  1,\ldots,p\right\}
$ and $\left\{  i,j\right\}  $ is an edge of $G_{p}$ if and only if $i-j$ is a
quadratic residue $\operatorname{mod}$ $p$. Paley graphs are conference
graphs, and their spectra are well known. A simple bound on the energy of
$G_{p}$ gives $\left\Vert G_{p}\right\Vert _{\ast}>p^{3/2}/2$ (see, e.g.,
\cite{Nik07j}).

Let $k:=\left(  p-1\right)  /2,$ and note that $k$ is even and $G_{p}$ is
$k$-regular. Theorem \ref{th3} implies that there exists a $k$-regular graph
$G$ of order $n$ with
\begin{align*}
\left\Vert G\right\Vert _{\ast}  &  >\left\Vert G_{p}\right\Vert _{\ast
}-3\sqrt{\left(  n-p\right)  kn}>\frac{1}{2}p^{3/2}-3\sqrt{n^{8/5}k/8}\\
&  >\frac{1}{2}\left(  n-n^{3/5}/8\right)  ^{3/2}-\frac{3}{4}n^{13/10}.
\end{align*}
On the other hand, using Bernoulli's inequality, we find that
\[
\frac{1}{2}\left(  n-n^{3/5}/8\right)  ^{3/2}-\frac{3}{4}n^{13/10}=\frac{1}%
{2}n^{3/2}\left(  1-n^{-2/5}/8\right)  ^{3/2}>\frac{1}{2}n^{3/2}-\frac{3}%
{32}n^{13/10}.
\]
Hence, $\left\Vert G\right\Vert _{\ast}>n^{3/2}/2-$ $n^{13/10},$ completing
the proof of Theorem \ref{th2}.
\end{proof}

\section{\label{EC}The energy of the complement of a graph}

It seems not widely known that the energy of a graph $G$ and the energy of its
complement $\overline{G}$ are quite close. Indeed, let $G$ be a graph of order
$n,$ let $J_{n\text{ }}$ be the all-ones square matrix of order $n,$ and let
$I_{n\text{\ }}$be the identity matrix of order $n.$ If $A$ and $\overline{A}$
are the adjacency matrices of $G$ and $\overline{G},$ then $A+\overline
{A}=J_{n}-I_{n},$ and using the triangle inequality for the trace norm, we
find that
\[
\left\Vert \overline{A}\right\Vert _{\ast}=\left\Vert J_{n}-I_{n}-A\right\Vert
_{\ast}\leq\left\Vert G\right\Vert _{\ast}+\left\Vert J_{n}-I_{n}\right\Vert
_{\ast}=\left\Vert G\right\Vert _{\ast}+2n-2.
\]
By symmetry, we get the following proposition:

\begin{proposition}
\label{pro1}If $G$ is a graph of order $n$ and $\overline{G}$ is the
complement of $G$, then
\begin{equation}
\left\vert \left\Vert \overline{G}\right\Vert _{\ast}-\left\Vert G\right\Vert
_{\ast}\right\vert \leq2n-2. \label{cin}%
\end{equation}
Equality in (\ref{cin}) holds if and only if $G$ or $\overline{G}$ is a
complete graph.
\end{proposition}

Instead of tackling the characterization for equality in (\ref{cin}) directly,
we shall prove a more elaborate bound, which implies this characterization
right away. Hereafter, we write $\lambda_{1}\left(  G\right)  ,\ldots
,\lambda_{n}\left(  G\right)  $ for the eigenvalues of the adjacency matrix of
$G$ arranged in descending order.

\begin{theorem}
\label{pro2}If $G$ is a graph of order $n$ and $\overline{G}$ is the
complement of $G$, then
\[
\left\Vert G\right\Vert _{\ast}-\left\Vert \overline{G}\right\Vert _{\ast}%
\leq2\lambda_{1}(G)
\]
and
\[
\left\Vert \overline{G}\right\Vert _{\ast}-\left\Vert G\right\Vert _{\ast}%
\leq2\lambda_{1}(\overline{G}).
\]

\end{theorem}

\begin{proof}
By definition we have
\begin{align*}
\left\Vert \overline{G}\right\Vert _{\ast}  &  =\lambda_{1}(\overline
{G})+|\lambda_{2}(\overline{G})|+\cdots+|\lambda_{n}(\overline{G})|,\\
\left\Vert G\right\Vert _{\ast}  &  =\lambda_{1}\left(  G\right)  +\left\vert
\lambda_{n}\left(  G\right)  \right\vert +\cdots+\left\vert \lambda_{2}\left(
G\right)  \right\vert .
\end{align*}
Thus, the triangle inequality for the absolute value implies that
\begin{align*}
\left\Vert \overline{G}\right\Vert _{\ast}-\lambda_{1}(\overline
{G})-\left\Vert G\right\Vert _{\ast}+\lambda_{1}\left(  G\right)   &
=\sum_{k=2}^{n}|\lambda_{k}(\overline{G})|-\left\vert \lambda_{n-k+2}\left(
G\right)  \right\vert \\
&  \leq\sum_{k=2}^{n}|\lambda_{k}(\overline{G})+\lambda_{n-k+2}\left(
G\right)  |.
\end{align*}
On the other hand, Weyl's inequalities for the eigenvalues of Hermitian
matrices (see, e.g., [5], p. 181) imply that
\[
\lambda_{k}\left(  G\right)  +\lambda_{n-k+2}(\overline{G})\leq\lambda
_{2}\left(  J_{n}-I_{n}\right)  =-1
\]
for any $k\in\left\{  2,\ldots,n\right\}  .$ Thus, we get
\begin{align*}
\sum_{k=2}^{n}|\lambda_{k}(\overline{G})+\lambda_{n-k+2}\left(  G\right)
|\text{ }  &  =-\sum_{k=2}^{n}\lambda_{k}(\overline{G})+\lambda_{n-k+2}\left(
G\right) \\
&  =\lambda_{1}(\overline{G})+\lambda_{1}\left(  G\right)  ,
\end{align*}
and the required inequalities follow.
\end{proof}

It seems that the bounds in Theorem \ref{pro2} can be improved, so we raise
the following problem.

\begin{problem}
Find the best possible upper bounds for $\left\Vert G\right\Vert _{\ast
}-\left\Vert \overline{G}\right\Vert _{\ast}$ for general and for regular graphs.
\end{problem}

\section{\label{ESR}The energy of some strongly regular graphs}

The goal of this section is to exhibit infinite classes of graphs for which
the bound (\ref{KM}) is exact or almost exact. Our first example uses the rich
class of symplectic graphs $Sp\left(  2m,q\right)  $ in the general form given
by Tang and Wan in \cite{TaWa06} (see the references of \cite{TaWa06} for
previous work on symplectic graphs). The graph $Sp\left(  2m,q\right)  $ is a
design graph and therefore forces equality in (\ref{KM}). Its complement
performs just slightly worse as seen in the next statement:

\begin{proposition}
\label{pro3}Let $q$ be a prime power.

(a) For every $n_{0}$ there exists a $k$-regular graph $G$ of order $n>n_{0}$
such that
\[
k=\frac{q-1}{q}n+\frac{1}{q},
\]
and
\[
\left\Vert G\right\Vert _{\ast}=k+\sqrt{k\left(  n-k\right)  \left(
n-1\right)  }.
\]

(b) For every $n_{0},$ there exists a $k$-regular graph $G$ of order $n>n_{0}$
such that
\[
k=\frac{n}{q}-\frac{q+1}{q},
\]
and
\[
\left\Vert G\right\Vert _{\ast}>\sqrt{k\left(  n-k\right)  n}-n+k+1.
\]

\end{proposition}

\begin{proof}
Recall that in \cite{TaWa06} Tang and Wan defined a class of strongly regular
graphs $Sp\left(  2m,q\right)  $ with parameters%
\[
\left(  \frac{q^{2m}-1}{q-1},q^{2m-1},q^{2m-2}\left(  q-1\right)
,q^{2m-2}\left(  q-1\right)  \right)  ,
\]
where $q$ is any prime power and $m$ is any positive integer.

Note that the eigenvalues of $Sq\left(  2m,q\right)  $ are $q^{2m-1},$
$q^{m-1},$ and $-q^{m-1}.$ Therefore, letting $G:=Sq\left(  2m,q\right)  ,$%
\[
n:=\frac{q^{2m}-1}{q-1}\text{ \ \ and \ \ }k:=q^{2m-1},
\]
we see that%
\[
k-\frac{q-1}{q}n=q^{2m-1}-\frac{q-1}{q}\frac{q^{2m}-1}{q-1}=\frac{1}{q},
\]
and
\[
\left\Vert G\right\Vert _{\ast}=q^{2m-1}+\left(  n-1\right)  q^{m-1}%
=k+\sqrt{k\left(  n-k\right)  \left(  n-1\right)  }.
\]
This observation proves (a). To prove (b) take $G$ to be the complement of
$Sp\left(  2m,q\right)  $. Now $G$ is $k$-regular with%
\[
k=n-\frac{q-1}{q}n-\frac{1}{q}-1=\frac{n}{q}-\frac{q+1}{q}.
\]
The eigenvalues of $G$ are:

- $k$ with multiplicity $1$

- $q^{m-1}-1$ with multiplicity $\left(  \left(  n-1\right)  +q^{m}\right)
/2$

-$-q^{m-1}-1$ with multiplicity $\left(  \left(  n-1\right)  -q^{m}\right)
/2$

Hence, the energy of $G$ satisfies
\begin{align*}
\left\Vert G\right\Vert _{\ast}  &  =\frac{\left(  n-1\right)  +q^{m}}%
{2}\left(  q^{m-1}-1\right)  +\frac{\left(  n-1\right)  -q^{m}}{2}\left(
q^{m-1}+1\right)  +n-q^{2m-1}-1\\
&  =\left\Vert Sp\left(  2m,q\right)  \right\Vert _{\ast}-q^{2m-1}%
-q^{m}+n-q^{2m-1}-1\\
&  =\sqrt{\left(  k+1\right)  \left(  n-k-1\right)  \left(  n-1\right)
}+k-\frac{q+1}{q}\left(  n-k-1\right)
\end{align*}

Two simple calculations show that%
\[
\sqrt{\left(  k+1\right)  \left(  n-k-1\right)  n}>\sqrt{k\left(  n-k\right)
n},
\]
and if $m\geq2,$ then
\[
k-\frac{q+1}{q}\left(  n-k-1\right)  >-\left(  n-k-1\right)  .
\]
Hence, for $m\geq2$ we see that
\[
\left\Vert G\right\Vert _{\ast}>\sqrt{k\left(  n-k\right)  n}-n+k+1,
\]
as claimed.
\end{proof}

The main implication of Proposition \ref{pro3} is the fact that the bound
(\ref{KM}) is exact or asymptotically exact for an infinite set of edge
densities of graphs. Among these densities are the numbers $1/5,$ $1/4,$
$1/3,$ $1/2,$ $2/3,$ $3/4,$ $4/5,$ etc. The fist unknown case is $1/6,$ so we
ask the following question:

\begin{question}
Is it true that for every $\varepsilon>0$ and $n_{0}>0,$ there exists a graph
$G$ of order $n>n_{0}$ such that
\[
\left(  \frac{1}{6}-\varepsilon\right)  \binom{n}{2}<e\left(  G\right)
<\left(  \frac{1}{6}+\varepsilon\right)  \binom{n}{2}%
\]
and
\[
\left\Vert G\right\Vert _{\ast}>\left(  \sqrt{5}/6-\varepsilon\right)
n^{3/2}.
\]

\end{question}

We conclude this section with a family of sparse graphs implying equality in
(\ref{KM}). In \cite{ArSz69} Ahrens and Szekeres constructed a family of
strongly regular graphs with parameters%
\[
\left(  q^{2}\left(  q+2\right)  ,q\left(  q+1\right)  ,q,q\right)  ,
\]
where $q$ is a prime power. These graphs give some credibility to Conjecture
\ref{con0}:

\begin{proposition}
For every $n_{0}$ there exists a $k$-regular graph $G$ of order $n>n_{0}$ such
that
\[
n^{2/3}-\frac{1}{3}n^{1/3}<k<n^{2/3},
\]
with
\[
\left\Vert G\right\Vert _{\ast}=k+\sqrt{k\left(  n-k\right)  \left(
n-1\right)  }.
\]

\end{proposition}

Using the results of Section \ref{DD}, one can deduce the following extension:

\begin{proposition}
For every $\varepsilon>0,$ and sufficiently large $n,$ if $k$ satisfies
\[
\left(  1-\varepsilon\right)  n^{2/3}<k<\left(  1+\varepsilon\right)
n^{2/3},
\]
and $kn$ is even, there exists a $k$-regular graph $G$ of order $n$ with%
\[
\left\Vert G\right\Vert _{\ast}\geq\left(  1-\varepsilon\right)  \sqrt{kn}.
\]

\end{proposition}

\section{\label{ERR}The energy of random regular graphs\ }

In \cite{CLL13}, Chen, Li and Lin studied the skew-energy of a random
$k$-regular oriented graph. Somewhat surprisingly, these authors missed the
fact that the similar methods apply to the energy of $k$-regular graphs. Thus,
in this section we fill in this void. In what follows, we use
\textquotedblleft almost any $k$-regular graph\textquotedblright\ as a synonym
of \textquotedblleft randomly chosen $k$-regular graph\textquotedblright.

\begin{theorem}
\label{th4}Let $k\geq2$ be a fixed integer. The energy of almost any
$k$-regular graph $G$ of order $n$ satisfies
\[
\left\Vert G\right\Vert _{\ast}=\frac{n}{\pi}\left(  2k\sqrt{k-1}-k\left(
k-2\right)  \arctan\frac{2\sqrt{k-1}}{k-2}\right)  +o\left(  n\right)  .
\]

\end{theorem}

\begin{proof}
Let $G$ be a randomly chosen $k$-regular graph of order $n$. In \cite{McK81}
McKay showed that, as $n$ increases, the empirical spectral distribution of
$G$ converges to the density function%
\[
f\left(  x\right)  =\left\{
\begin{array}
[c]{ll}%
\frac{k\sqrt{4\left(  k-1\right)  -x^{2}}}{2\pi\left(  k^{2}-x^{2}\right)
}, & \text{if }\left\vert x\right\vert \leq2\sqrt{k-1},\text{ }\\
0, & \text{otherwise.}%
\end{array}
\right.
\]
This implies that the energy of $G$ almost surely satisfies
\[
\left\Vert G\right\Vert _{\ast}=n\int\limits_{-2\sqrt{k-1}}^{2\sqrt{k-1}%
}\left\vert x\right\vert \frac{k\sqrt{4\left(  k-1\right)  -x^{2}}}%
{2\pi\left(  k^{2}-x^{2}\right)  }dx+o\left(  n\right)  .
\]
After a change of variable, the indefinite integral can be calculated, and we
find that, almost surely, $\left\Vert G\right\Vert _{\ast}$ satisfies
\[
\left\Vert G\right\Vert _{\ast}=\frac{n}{\pi}\left(  2k\sqrt{k-1}-k\left(
k-2\right)  \arctan\frac{2\sqrt{k-1}}{k-2}\right)  +o\left(  n\right)  ,
\]
which completes the proof.
\end{proof}

Note that the energy of a randomly chosen $k$-regular graph of order $n$ is
almost equal to the skew energy of a randomly chosen $k$-regular oriented
graph of order $n$ (Theorem 4.3 of \cite{CLL13}). This is not very surprising,
as $k$-regular graphs of large order locally are trees, and the skew energy of
oriented trees is equal to the energy of the underlying unoriented tree.

By some involved calculations one can show that%
\[
\frac{8}{3}\sqrt{k}<2k\sqrt{k-1}-k\left(  k-2\right)  \arctan\frac{2\sqrt
{k-1}}{k-2}<\frac{8}{3}\sqrt{k-1}\left(  1+\frac{1}{k}\right)  .
\]
Hence, as $n$ increases, the energy of almost any $k$-regular graph $G$ of
order $n$ satisfies%
\[
\frac{8}{3\pi}\sqrt{k}n<\left\Vert G\right\Vert _{\ast}<\frac{8}{3\pi}%
\sqrt{k-1}\left(  1+\frac{1}{k}\right)  n.
\]

Let us reiterate that in the above discussion $k$ is fixed, and $n$ tends to
$\infty.$ On the other hand, the distribution of the eigenvalues of random
$k$-regular graphs whenever $k\rightarrow\infty$ with $n$ has been found only
recently, by Dumitriu and Pal \cite{DuPa13}, and Tran, Vu, and Wang
\cite{TVW13}. Based on the latter work, we prove the following theorem.

\begin{theorem}
\label{th5}Let $k\rightarrow\infty$ with $n.$ The energy of almost any
$k$-regular graph $G$ of order $n$ satisfies
\[
\left\Vert G\right\Vert _{\ast}=\left(  \frac{8}{3\pi}+o\left(  1\right)
\right)  \sqrt{k\left(  n-k\right)  n}.
\]

\end{theorem}

\begin{proof}
Let $G_{n,k}$ be a randomly chosen $k$-regular graph of order $n,$ let $A_{n}$
be its adjacency matrix, and let
\[
M_{n}:=\frac{1}{\sqrt{\frac{k}{n}\left(  1-\frac{k}{n}\right)  }}\left(
A-\frac{k}{n}J_{n}\right)
\]
A recent result of Tran, Vu, and Wang \cite{TVW13} states that if $k$
$\rightarrow\infty$ with $n,$ then the empirical spectral distribution of the
matrix $n^{-1/2}M_{n},$ converges to the standard semicircle distribution,
i.e., to the density function%
\[
f\left(  x\right)  =\left\{
\begin{array}
[c]{ll}%
\frac{1}{2\pi}\sqrt{4-x^{2}}, & \text{if }\left\vert x\right\vert \leq2,\text{
}\\
0, & \text{otherwise.}%
\end{array}
\right.
\]
First, let us note that
\[
\left\Vert A\right\Vert _{\ast}=\sqrt{\frac{k}{n}\left(  1-\frac{k}{n}\right)
}\left\Vert M\right\Vert _{\ast}+k.
\]
For $\left\Vert M\right\Vert _{\ast}$ we get
\begin{align*}
\left\Vert M\right\Vert _{\ast} &  =\frac{n^{3/2}}{2\pi}\int\limits_{-2}%
^{2}\left\vert x\right\vert \sqrt{4-x^{2}}+o\left(  n^{3/2}\right)
=\frac{n^{3/2}}{\pi}\int\limits_{0}^{2}x\sqrt{4-x^{2}}+o\left(  n^{3/2}%
\right)  \\
&  =\frac{8}{3\pi}n^{3/2}+o\left(  n^{3/2}\right)  .
\end{align*}
Hence,
\[
\left\Vert A\right\Vert _{\ast}=\left(  \frac{8}{3\pi}+o\left(  1\right)
\right)  \sqrt{k\left(  n-k\right)  n},
\]
as claimed.
\end{proof}

Comparing Theorem \ref{th4} with \ref{th1} and Theorem \ref{th5} with bound
(\ref{KM}), we see that the energy of almost all $k$-regular graphs is more
than $84\%$ of the maximum one. This high value adds some extra credibility to
Conjecture \ref{mcon}.


\bigskip

\begin{thebibliography}{99}                                                                                               %


\bibitem {ArSz69}R.W. Ahrens and G. Szekeres, On a combinatorial
generalization of 27 lines associated with a cubic surface, \emph{J. Austral.
Math. Soc. }\textbf{10} (1969), 485-492.

\bibitem {BHP97}R. C. Baker, G. Harman, J. Pintz, The exceptional set for
Goldbach's problem in short intervals, \emph{Sieve methods, exponential sums,
and their applications in number theory }(Cardiff, 1995),\ pp. 1--54, LMS
Lecture Notes Ser., 237, Cambridge Univ. Press, Cambridge, 1997.

\bibitem {Bal04}R. Balakrishnan, The energy of a graph, \emph{Linear Algebra
Appl.} \textbf{387} (2004), 287--295.

\bibitem {BJK99}T. Beth, D. Jungnickel, and H. Lenz, Design Theory, vol. 1,
2nd ed., \emph{Cambridge, Cambridge University Press}, Cambridge,1999.

\bibitem {CLL13}X. Chen, X. Li, and H. Lin, The skew energy of random oriented
graphs, \emph{Linear Algebra Appl. }\textbf{438 }(2013), 4547--4556.

\bibitem {DHK14}E.R. van Dam, W.H. Haemers, and J.H Koolen, Regular graphs
with maximal energy per vertex, \emph{J. Combin. Theory Ser B }\textbf{107
}(2014), 123--131.

\bibitem {DuPa13}I. Dumitriu and S. Pal, Sparse regular random graphs:
Spectral density and eigenvectors, \emph{Ann. Probab.} \textbf{40} (2012), 2197--2235.

\bibitem {Gut78}I. Gutman, The energy of a graph, \emph{Ber. Math.-Stat. Sekt.
Forschungszent. Graz} \textbf{103} (1978), 1--22.

\bibitem {GLS12}I. Gutman, X. Li, and Y. Shi, \emph{Graph Energy}, New York,
Springer, 2012, 266 pp.

\bibitem {GFPR07}I. Gutman, S.Z. Firoozabadi, J.A. de la Pe\~{n}a. and J.
Rada, On the energy of regular graphs, \emph{MATCH} \textbf{57} (2007), 435--442.

\bibitem {HoJo85}R. Horn and C. Johnson, Matrix Analysis, \emph{Cambridge
University Press}, Cambridge, 1985. xiii+561pp.

\bibitem {KoMo01}J.H. Koolen and V. Moulton, Maximal energy graphs, \emph{Adv.
Appl. Math.}\textbf{ 26 }(2001), 47--52.

\bibitem {LLS10}X. Li , Y. Li, Y. Shi, Note on the energy of regular graphs,
\emph{Linear Algebra Appl. }\textbf{432} (2010), 1144--1146

\bibitem {McK81}B.D. McKay, The expected eigenvalue distribution of a large
regular graph, \emph{Linear Algebra Appl.} \textbf{40} (1981), 203--216Ann. Prob

\bibitem {Nik06}V. Nikiforov, Eigenvalues and degree deviation in graphs,
\emph{Linear Algebra Appl.} \textbf{414} (2006), 347-360.

\bibitem {Nik07i}V. Nikiforov, The energy of graphs and matrices, \emph{J.
Math. Anal. Appl.} \textbf{326} (2007), 1472--1475.

\bibitem {Nik07j}V. Nikiforov, Graphs and matrices with maximal energy,
\emph{J. Math. Anal. Appl. }\textbf{327} (2007), 735--738.

\bibitem {TaWa06}Z. Tang and Z.-X. Wan, Symplectic graphs and their
automorphisms, \ \emph{Europ. J. Combin. }\textbf{ 27 }(2006) 38--50.

\bibitem {TVW13}L.V. Tran, V.H. Vu, and K. Wang, Sparse random graph:
eigenvalues and eigenvectors, \emph{Random Struct. Alg.} \textbf{42 }(2013), 110--134.
\end{thebibliography}
\end{document}